\documentclass[12pt,a4paper]{article}
\usepackage{anziamjedraft}

\usepackage{amsmath,amssymb,amsfonts} 
\usepackage{bm}
\usepackage{multirow}\usepackage{hhline}
\usepackage[width=14cm]{geometry}
\usepackage{color}
\usepackage{lineno}

\newcommand{\vertiii}[1]{{\left\vert\kern-0.25ex\left\vert\kern-0.25ex\left\vert #1 \right\vert\kern-0.25ex\right\vert\kern-0.25ex\right\vert}}

\DeclareMathOperator*{\argmin}{\text{argmin}}

\title{A three-field formulation of the Poisson problem with Nitsche approach}
\author[M.]{Ilyas}
\address{School of Mathematical and Physical Sciences, University of Newcastle, Callaghan, NSW 2308, \textsc{Australia}}
\mailto{muhammad.ilyas@uon.edu.au}
\myorcid{0000-0002-5298-756X}
\author[B.P.]{Lamichhane}
\address{School of Mathematical and Physical Sciences, University of Newcastle, Callaghan, NSW 2308, \textsc{Australia}}


\begin{document}
	
	\maketitle
	\abstract{We modify a three-field formulation of the Poisson problem with Nitsche approach for approximating Dirichlet boundary conditions. Nitsche approach allows us to weakly impose Dirichlet boundary condition but still preserves the optimal convergence. We use the biorthogonal system for efficient numerical computation and introduce a stabilisation term so that the problem is coercive on the whole space. Numerical examples are presented to verify the algebraic formulation of the problem.
	}
	\tableofcontents

\section{Introduction}

The finite element method is a powerful and efficient method to handle complicated geometries and impose the associated boundary conditions. However, in some cases, the treatment of the Dirichlet-type boundary conditions compromise the stability and accuracy of the standard finite element method \cite{JS09}.

In order to relax the Dirichlet boundary condition constraint, we need to modify the standard finite element approach. Generally, we can do this by imposing the Dirichlet boundary condition as a penalty term \cite{Aub72, Bab73}. One of such methods is Nitsche's method \cite{Nit71}, which imposes the Dirichlet boundary condition weakly in the formulation without the need of a Lagrange multiplier. Moreover, compared to other penalty method, Nitsche's method adds the consistency, symmetry and stability terms so that this method can achieve optimal convergence. There are so many applications of Nitsche's method in many areas, such as elasticity \cite{BBH09}, interface problems \cite{Han05}, potential flows \cite{JGS16} and plasticity \cite{Tru16}. 

In this article, we modify a mixed finite element method, based on the three-field formulation \cite{IL16}, with Nitsche approach to solve a Poisson problem. A similar three-field formulation, known as Hu-Washizu formulation, is popular in linear elasticity field \cite{LMR13}. The three-field formulation allows us to apply a biorthogonal system which leads to a very efficient finite element method. In order to overcome the difficulty of coercivity condition, we introduce a stabilisation term \cite{IL16} of the associated bilinear form so that it is coercive on the whole space.


The structure of the article is as follows. In the next section we recall the Nitsche formulation for Poisson problem and introduce a three-field formulation with this approach. We modify the three-field formulation to include a stabilisation term. We introduce the finite element approximation and prove the well-posedness condition in Section \ref{sec:fem}. We then show the algebraic formulation and a priori error estimate in Section \ref{sec:algest}. Two numerical examples are presented in Section \ref{sec:examples}. Finally, a short conclusion is written in Section \ref{sec:conclusion}.

\section{A Three-field Formulation for Poisson Problem}

\label{sec:formulation}

\subsection*{Sobolev Spaces}

Let $V=H^{1}\left( \Omega \right) $ and $L=\left[ L^{2}\left( \Omega \right)\right] ^{2}$. The Sobolev spaces $H^{k}\left( S\right) $ for $S\subset\Omega $ or $S\subset \Gamma $, and $k\geq 0$ are defined in the standard way \cite{Cia78}. We introduce the space $H^{-1/2}\left(\Gamma \right) $, the dual space of $H^{1/2}\left( \Omega \right) $, with the norm
\begin{equation*}
\left\Vert \mu \right\Vert _{-1/2,\Gamma }=\sup_{z\in H^{1/2}\left( \Gamma \right) }\frac{\left\langle \mu ,z\right\rangle }{\left\Vert z\right\Vert_{1/2,\Gamma }}
\end{equation*}
where $\left\langle \cdot ,\cdot \right\rangle $ denotes the duality pairing. For functions $v\in H^{1}\left( \Omega \right) $ with $\Delta v\in L^{2}\left( \Omega \right) $, it holds \cite{Bab73} $\frac{\partial v}{\partial n}\in H^{-1/2}\left( \Gamma \right) $ with
\begin{equation*}
\left\Vert \frac{\partial u}{\partial n}\right\Vert _{-1/2,\Gamma }\leq C\left( \left\Vert v\right\Vert _{1}+\left\Vert \Delta v\right\Vert_{0}\right) .
\end{equation*}

We will also introduce the mesh-dependent norms
\begin{eqnarray*}
\left\Vert v\right\Vert _{1/2,h}^{2} &=&\sum \frac{1}{h_{e}}\left\Vert v\right\Vert _{0,e}^{2}\text{ for }v\in H^{1}\left( \Omega \right),  \\
\left\Vert z\right\Vert _{-1/2,h}^{2} &=&\sum h_{e}\left\Vert z\right\Vert_{0,e}^{2}\text{ for }z\in L^{2}\left( \Gamma \right), 
\end{eqnarray*}%
and for these norms it holds
\begin{equation}
\left\langle v,z\right\rangle \leq \left\Vert v\right\Vert_{1/2,h}\left\Vert z\right\Vert _{-1/2,h}\text{ for }\left( v,z\right) \in H^{1}\left( \Omega \right) \times L^{2}\left( \Gamma \right).
\label{ineq:dualpair}
\end{equation}

For the rest of the article, we denote
\begin{equation*}
\left\Vert u\right\Vert _{1,h}=\left\Vert u\right\Vert _{1,\Omega}+\left\Vert u\right\Vert _{1/2,h}\text{ for }u\in H^{1}\left(\Omega\right).
\end{equation*}

\subsection*{Nitsche Formulation for the Poisson Problem}

The mixed formulation is obtained by introducing $\sigma =\nabla u$. Given $f\in L^{2}\left( \Omega \right) $, the (Nitsche) minimisation problem can be written as
\begin{equation}
\argmin_{\substack{ \left( u,\sigma \right) \in V\times L \\ \sigma=\nabla u}}\frac{1}{2}\left\Vert \sigma \right\Vert _{0,\Omega }^{2}+\frac{\alpha }{2}\left\Vert u-g_{D}\right\Vert _{1/2,h}^{2}-\left\langle \sigma\cdot \mathbf{n},v-g_{D}\right\rangle -\int_{\Omega }fu \,dx.  \label{eq:minmixed}
\end{equation}

We write a variational equation for $\sigma =\nabla u$ using the Lagrange multiplier space $M=L$ to obtain the saddle-point problem of the minimisation problem \eqref{eq:minmixed}.
The saddle point formulation is to find $\left( u,\sigma ,\varphi \right) \in V\times L\times M$ such that
\begin{equation}
\begin{array}{lcll}
\tilde{a}\left[ \left( u,\sigma \right) ,\left( v,\tau \right) \right] +b\left[ \left( v,\tau \right) ,\varphi \right]  &=&\ell \left( v,\tau\right) ,&\left( v,\tau \right) \in V\times L,  \notag \\
b\left[ \left( u,\sigma \right) ,\psi \right]  &=&0,&\psi \in M,
\end{array}
\label{eq:original}
\end{equation}
where
\begin{eqnarray*}
\tilde{a}\left[ \left( u,\sigma \right) ,\left( v,\tau \right) \right] &=&\int_{\Omega }\sigma \cdot \tau \,dx+\alpha \left\langle u,v\right\rangle_{1/2,h}-\left\langle \sigma \cdot \mathbf{n},v\right\rangle -\left\langle\tau \cdot \mathbf{n},u\right\rangle , \\
b\left[ \left( u,\sigma \right) ,\psi \right]  &=&\int_{\Omega }\left(\sigma -\nabla u\right) \psi \,dx, \\
\ell \left( v,\tau\right)&=&\int_{\Omega }fv\,dx-\left\langle \tau \cdot \mathbf{n},g_{D}\right\rangle +\alpha\left\langle g_{D},v\right\rangle _{1/2,h}.
\end{eqnarray*}%
where $\left\langle \cdot ,\cdot \right\rangle $ denotes duality pairing between $H^{1/2}\left( \Omega \right) $ and $H^{-1/2}\left( \Gamma \right) $.

\section{Finite Element Discretisation}
\label{sec:fem}

Let $\mathcal{T}_{h}$ be a quasi-uniform triangulation of the polygonal domain $\Omega $. We use the standard linear finite element space $V_{h}\subset H^{1}\left( \Omega \right) $ defined on the triangulation $\mathcal{T}_{h}$, where 
\begin{equation*}
V_{h}:=\{v\in C^{0}\left( \Omega \right) :v|_{T}\in \mathcal{P}_{1}\left(T\right) ,T\in \mathcal{T}_{h}\}.
\end{equation*}

The finite element space for the gradient of the solution is $L_{h}=\left[ V_{h}\right]^{2}$. Let $\{\rho _{1},\rho _{2},\dots ,\rho _{N}\}$ be the finite element basis for $V_{h}$. Starting with the standard basis for $V_{h}$, we construct a space $Q_{h}$ spanned by the basis $\{\mu _{1},\mu _{2},\dots,\mu _{N}\}$ so that the basis functions of $V_{h}$ and $Q_{h}$ satisfy the biorthogonality condition 
\begin{equation*}
\int_{\Omega }\rho _{i}\mu _{j}\,dx=c_{j}\delta _{ij},\quad c_{j}\neq 0,\quad 1\leq i,j,\leq N,
\end{equation*}
where $\delta _{ij}$ is the Kronecker symbol, and $c_{j}$ a scaling factor. Therefore, the sets of basis functions of $V_{h}$ and $Q_{h}$ form a biorthogonal system. The basis functions of $Q_{h}$ are constructed locally on a reference element $\hat{T}$ so that the basis functions of $V_{h}$ and $Q_{h}$ have the same support, and in each element the sum of all the basis functions of $Q_{h}$ is one \cite{LMR13}. We let $M_{h}=\left[ Q_{h}\right] ^{2}$, thus our problem is to find $\left( u_{h},\sigma _{h},\varphi_{h}\right) \in V_{h}\times L_{h}\times M_{h}$ such that 
\begin{equation}
\begin{array}{lcll}
\tilde{a}\left[ \left( u_{h},\sigma _{h}\right) ,\left( v_{h},\tau _{h}\right)\right] +b\left[ \left( v_{h},\tau _{h}\right) ,\varphi _{h}\right] &=&\ell\left( v_{h}, \tau_h\right) ,&\left( v_{h},\tau _{h}\right) \in V_{h}\times L_{h}, \\
b\left[ \left( u_{h},\sigma _{h}\right) ,\psi _{h}\right] &=&0,&\psi_{h}\in M_{h}.
\end{array}
\label{eq:algebraic}
\end{equation}
To show that the saddle-point problem has a unique solution, we need to show
that the following well-posedness conditions are satisfied.

\begin{enumerate}
\item The linear form $\ell \left( \cdot \right) $, the bilinear forms $%
\tilde{a}\left[ \cdot ,\cdot \right] $ and $b\left[ \cdot ,\cdot \right] $
are continuous on the spaces in which they are defined.

\item The bilinear form $\tilde{a}\left[ \cdot ,\cdot \right] $ is coercive
on the kernel space $K_h$ defined as%
\begin{equation*}
K_h=\left\{ \left( u_h,\sigma_h \right) \in V_h\times L_h:b\left[ \left( u_h,\sigma_h
\right) ,\psi_h \right] =0,~\text{for all }\psi_h \in M_h\right\} .
\end{equation*}

\item The bilinear form $b\left[ \cdot ,\cdot \right] $ satisfies the 
\textit{inf-sup }condition%
\begin{equation*}
\inf_{\psi_h \in M_h}\sup_{\left( v_h,\tau_h \right) \in V_h\times L_h}\frac{b\left[
\left( v_h,\tau_h \right) ,\psi_h \right] }{\left\Vert v_h,\tau_h \right\Vert
_{V_h\times L_h}\left\Vert \psi_h \right\Vert _{0,\Omega}}\geq \gamma ,~\gamma >0.
\end{equation*}
\end{enumerate}

The mesh-dependent norm for the product space $V_h\times L_h$ is defined by
\begin{equation*}
\left\Vert u_h,\sigma_h \right\Vert _{V_h\times L_h}^{2}=\left\Vert u_h\right\Vert _{1,h}^{2} 
+ \left\Vert \sigma_h\right\Vert _{0,\Omega }^{2}, \quad \left( u_h,\sigma_h \right) \in V_h\times L_h.
\end{equation*}


With the introduction of $M_h$, the bilinear form $\tilde{a}\left[ \cdot,\cdot \right] $ is not coercive on the kernel subspace $K_h\subset V_h\times L_h$. Thus, we need to modify the bilinear form $\tilde{a}\left[ \cdot,\cdot \right] $ so that it is coercive on the kernel space $K_h$ or even the whole space $V_h\times L_h$. In this article, we modify the bilinear form $\tilde{a}\left[ \cdot ,\cdot \right] $ by adding a stabilisation term so that it is coercive on the whole space $V_h\times L_h$ \cite{IL16}
\begin{eqnarray*}
a\left[ \left( u_h,\sigma_h \right) ,\left( v_h,\tau_h \right) \right]&=& r\int_{\Omega }\sigma_h \cdot \tau_h \,dx+\left( 1-r\right) \int_{\Omega }\nabla u_h\cdot \nabla v_h\,dx\\
& & +\alpha \left\langle u_h,v_h\right\rangle_{1/2,h}-\left\langle \sigma_h \cdot \mathbf{n},v_h\right\rangle -\left\langle\tau_h \cdot \mathbf{n},u_h\right\rangle,
\end{eqnarray*}
for $0<r<1$. 

We use the following inverse estimate result \cite{JS09} to show the continuity condition of $\ell \left( \cdot \right) $ and also continuity and coercivity condition of the bilinear form $a\left[ \cdot ,\cdot \right] $,
\begin{equation}
C_{I}\left\Vert \frac{\partial v_h}{\partial n}\right\Vert _{-1/2,h}\leq\left\Vert \nabla v_h\right\Vert _{0,\Omega }\text{ for }v_h\in V_{h}.
\label{ineq:inverse}
\end{equation}
The continuity of the linear form $\ell \left( \cdot \right) $, and the bilinear forms $a\left[ \cdot ,\cdot \right] $ and $b\left[ \cdot ,\cdot \right] $ then follows from the Cauchy-Schwarz inequality, the duality pairing \eqref{ineq:dualpair} and the inverse estimate \eqref{ineq:inverse}.

For the coercivity condition, using the inverse estimate \eqref{ineq:inverse} and the following Poincare-Friedrichs inequality,
\begin{equation*}
\left\Vert u_h\right\Vert _{1,\Omega }^{2}=\left\Vert u_h\right\Vert _{0,\Omega}^{2}+\left\Vert \nabla u_h\right\Vert _{0,\Omega }^{2}\leq \left(c^{2}+1\right) \left\Vert \nabla u_h\right\Vert _{0,\Omega }^{2},
\end{equation*}
we can write
\begin{equation}
\begin{aligned}
&\left\vert a\left[ \left( u_h,\sigma_h \right) ,\left( u_h,\sigma_h \right) \right]\right\vert&\\
&=r\left\Vert \sigma_h \right\Vert _{0,\Omega }^{2}+\left(1-r\right) \left\Vert \nabla u_h\right\Vert _{0,\Omega }^{2}+\alpha \left\Vert u_h\right\Vert _{1/2,h}^{2}-2\left\langle \sigma_h \cdot \mathbf{n},u_h\right\rangle , &\\
&\geq r\left\Vert \sigma_h \right\Vert _{0,\Omega }^{2}+\left( 1-r\right)\left\Vert \nabla u_h\right\Vert _{0,\Omega }^{2}-2\left\Vert \sigma_h \cdot \mathbf{n}\right\Vert _{-1/2,h }\left\Vert u_h\right\Vert _{1/2,h}+\alpha \left\Vert u_h\right\Vert _{1/2,h}^{2}, &\\
&\geq r\left\Vert \sigma_h \right\Vert _{0,\Omega }^{2}+\left( 1-r\right)\left\Vert \nabla u_h\right\Vert _{0,\Omega }^{2}-\left( \frac{1}{\varepsilon }\left\Vert \sigma_h \cdot \mathbf{n}\right\Vert _{-1/2,h}^{2}+\varepsilon \left\Vert u_h\right\Vert _{1/2,h }^{2}\right) +\alpha\left\Vert u_h\right\Vert _{1/2,h}^{2}, &\\
&\geq \left( r-\frac{1}{\varepsilon C_{I}}\right) \left\Vert \sigma_h\right\Vert _{0,\Omega }^{2}+\left( 1-r\right) \left\Vert \nabla u_h\right\Vert _{0,\Omega }^{2}+\left( \alpha -\varepsilon \right) \left\Vert u_h\right\Vert _{1/2,h}^{2}, &\\
&\geq \left( r-\frac{1}{\varepsilon C_{I}}\right) \left\Vert \sigma_h\right\Vert _{0,\Omega }^{2}+\frac{1-r}{c^{2}+1}\left\Vert u_h\right\Vert_{1,\Omega }^{2}+\left( \alpha -\varepsilon \right) \left\Vert u_h\right\Vert_{1/2,h}^{2}, &\\
&\geq C\left\Vert \left( u_h,\sigma_h \right) \right\Vert _{V_h\times L_h}^{2},&
\end{aligned}
\label{ineq:coercivity}
\end{equation}
where $C$ is the minimum of $\left( r-\frac{1}{\varepsilon C_{I}}\right)$, $\frac{1-r}{c^{2}+1}$ and $\left( \alpha -\varepsilon \right)$. We also require $\frac{1}{C_{I}}<r\varepsilon <\varepsilon<\alpha $ and $0<r<1$. From this point forward, we use constant $C$ as a mesh-independent generic constant. 

Now the \textit{inf-sup }condition for the bilinear form $b\left[ \cdot,\cdot \right] $ can be shown as in \cite{IL16}.
Thus we have proved the following theorem.

\begin{theorem}
\label{thm:solution}The saddle point problem \eqref{eq:algebraic} with stabilised $a\left[\cdot,\cdot\right]$ has a unique solution $\left( u_h,\sigma_h ,\varphi_h \right) \in V_h\times L_h\times M_h$. The solution also satisfies $\left\Vert \left( u_h,\sigma_h \right) \right\Vert _{V_h\times L_h} + \left\Vert \varphi_h \right\Vert_{0,\Omega} \leq C\left\Vert f \right\Vert_{0,\Omega}.$
\end{theorem}

\section{Algebraic Formulation}

\label{sec:algest}

In order to present an algebraic formulation of the problem, we use $\left(x_{u},x_{\sigma },x_{\varphi }\right) $ for the vector representation of thesolution $\left( u_{h},\sigma _{h},\varphi _{h}\right) $ as elements in $V_{h}\times L_{h}\times M_{h}$. Let $S$, $D$, $A$, $B$, $C$ and $M$ be the matrices associated with bilinear forms $\int_{\Omega }\nabla u_{h}\cdot \nabla v_{h}\,dx,~\int_{\Omega }\tau _{h}\cdot \varphi _{h}\,dx,~\int_{\Gamma}\left( \sigma _{h}\cdot \mathbf{n}\right) u_{h}~ds,$ $\int_{\Omega}\nabla v_{h}\cdot\varphi _{h}\,dx,~\sum \frac{1}{h_{e}}\int_{e}u_{h}v_{h}~ds $ and $\int_{\Omega }\sigma _{h}\cdot \tau _{h}\,dx$, respectively. For the right hand side, we write $\mathbf{f}_1$ and $\mathbf{f}_2$ to represent the discrete forms of $\int_{\Omega }f\,v_h\,dx\,v_h\,ds +\alpha \left\langle g_{D},v_h\right\rangle _{1/2,h}$ and $\left\langle \tau_h \cdot \mathbf{n},g_{D}\right\rangle$, respectively. Then the algebraic formulation of the problem is 
\begin{equation}
\left[ 
\begin{array}{ccc}
\left( 1-r\right) S+\alpha C & -A & -B \\ 
-A^{T} & rM & D \\ 
-B^{T} & D & 0
\end{array}%
\right] \left[ 
\begin{array}{c}
x_{u} \\ 
x_{\sigma } \\ 
x_{\varphi }
\end{array}
\right] =\left[ 
\begin{array}{c}
\mathbf{f}_1 \\ 
-\mathbf{f}_2 \\ 
\mathbf{0}
\end{array}
\right] ,  
\label{matrix1st}
\end{equation}%
where the first two equations of \eqref{matrix1st} correspond to first equation of \eqref{eq:algebraic} with stabilised $a\left[\cdot,\cdot\right]$, by setting $\sigma _{h}=0$ and $v_{h}=0$, respectively. After statically condensing out degrees of freedom associated with $\sigma_h$ and $\phi_h$ in \eqref{matrix1st}, we arrive at the following system
\[
Kx_u = F
\]
where
\begin{eqnarray*}
K & = & \left(1-r\right)S + \alpha C - AD^{-1}B^{T} - BD^{-1}A^{T} + rBD^{-1}MD^{-1}B^{T},\\
F & = & \mathbf{f}_1 - BD\mathbf{f}_2.
\end{eqnarray*}
Due to the choice of a biorthogonal system, matrix $D$ is diagonal. As a result, the statically condensed system matrix is sparse.

We introduce two projections $P_h:L^2\left(\Omega\right) \rightarrow Q_h$ and $P^{\ast}_h:L^2(\Omega) \rightarrow V_h$ as follows for $v \in L^2(\Omega)$.
\begin{equation}
\int\left(P_h v-v\right)\cdot\mu_{h}\,dx=0,\quad\mu_{h}\in Q_{h},\\
\int\left(P^{\ast}_h v-v\right)\cdot\varphi_{h}\,dx=0,\quad\varphi_{h}\in V_{h}.\\
\end{equation}

They satisfy the following estimates for $u \in H^1(\Omega)$ :
\begin{equation}
\Vert P_h u - u \Vert_{0,\Omega} \leq Ch \Vert u \Vert_{1,\Omega}, \quad 
\Vert P^{\ast}_h u - u \Vert_{0,\Omega} \leq Ch \Vert u \Vert_{1,\Omega}. 
\label{ineq:Pestimate}
\end{equation}
Using this projection, our problem is to find $u_{h}\in V_{h}$ such that,
\begin{equation}
A\left(u_{h},v_{h}\right)=L\left(v_{h}\right),~v_{h}\in V_{h}
\label{eq:condensed}
\end{equation}
where
\begin{equation}
\begin{aligned}
A\left(u_{h},v_{h}\right) & =\int_{\Omega}P_h\left(\nabla u_{h}\right)\cdot P_h\left(\nabla v_{h}\right)\,dx+\alpha\left\langle u_{h},v_{h}\right\rangle _{1/2,h}\\
&\quad -\int_{\Gamma}\left(P_h\left(\nabla u_{h}\right)\cdot\mathbf{n}\right)v_{h}~ds-\int_{\Gamma}\left(P_h\left(\nabla v_{h}\right)\cdot\mathbf{n}\right)u_{h}~ds,\\
L\left(v_{h}\right)  & =\int_{\Omega}fv_{h}\,dx-\int_{\Gamma}\left(P_h\left(\nabla v_{h}\right)\cdot n\right)g_{D}~ds+\alpha\left\langle g_{D},v_{h}\right\rangle _{1/2,h}.
\end{aligned}
\end{equation}
We also introduce two mesh-dependent norms
\begin{equation*}
\begin{array}{ll}
\left\Vert u_{h}\right\Vert _{h}^{2}  =\left\Vert u_{h}\right\Vert_{1,h}^{2}+\left\Vert P_{h}\left(\nabla u_{h}\right)\right\Vert_{0,\Omega}^{2}, & u_h \in V_h,\\
\vertiii{u}_{h}^{2}  =\left\Vert u\right\Vert _{1,h}^{2}+\left\Vert P_{h}\left(\nabla u\right)\right\Vert_{0,\Omega}^{2}+\left\Vert \nabla u\cdot n\right\Vert _{-1/2,h}^{2}, & u \in H^2\left(\Omega\right),
\end{array}
\end{equation*}
so that
\begin{equation}
\left\vert A\left(u,v_{h}\right)\right\vert \leq\left\vert \left\Vert u\right\Vert \right\vert _{h}\left\Vert v_{h}\right\Vert _{h},\quad u\in V\text{ and }v_{h}\in V_{h}.
\label{ineq:continuity}
\end{equation}
We get the following estimate combining the interpolation estimate of Lemma 3.4 of \cite{JS09}.
\begin{equation}\label{lemma1}
\inf_{v_h \in V_{h}}\vertiii{u-v_h}_{h}\leq
Ch\left\Vert u \right\Vert _{2,\Omega}.
\end{equation}
We then have the following theorem.

\begin{theorem}
\label{thm:apriori}Let $u_{h}\in V_{h}$ be the solution to the problem \eqref{eq:condensed}. Suppose that $u\in H^{2}\left( \Omega \right) $ is the solution to the problem \eqref{eq:original} then
\begin{equation*}
\left\Vert u-u_{h}\right\Vert _{h}\leq Ch\left\Vert u \right\Vert _{2,\Omega}.
\end{equation*}
\end{theorem}


\begin{proof}
	From the coercivity \eqref{ineq:coercivity} and continuity condition \eqref{ineq:continuity},
	\begin{equation*}
	\begin{aligned}
	\alpha\left\Vert u_{h}-v_{h}\right\Vert _{h}^{2}  & \leq A\left(u_{h}-v_{h},u_{h}-v_{h}\right),\\
	& =A\left(u-v_{h},u_{h}-v_{h}\right)-A\left(u,u_{h}-v_{h}\right)+A\left(u_{h},u_{h}-v_{h}\right),\\
	& =A\left(u-v_{h},u_{h}-v_{h}\right)-A\left(u,u_{h}-v_{h}\right)+L\left(u_{h}-v_{h}\right),\\
	& \leq\left\vert \left\Vert u-v_{h}\right\Vert \right\vert _{h}\left\Vert u_{h}-v_{h}\right\Vert _{h}+L\left(u_{h}-v_{h}\right)-A\left(u,u_{h}-v_{h}\right).
	\end{aligned}
	\end{equation*}
	Using $w_{h}=u_{h}-v_{h}$ and divide both sides by $\left\Vert w_{h}\right\Vert_{h}$, we get
	\begin{equation*}
	\alpha\left\Vert u_{h}-v_{h}\right\Vert _{h}\leq\vertiii{u-v_{h}}_{h}+\frac{L\left(w_{h}\right)-A\left(u,w_{h}\right)}{\left\Vert w_{h}\right\Vert _{h}}.
	\end{equation*}
	Following exactly as in the proof of Strang's second lemma \cite{Bra07} we get
	\begin{equation}
	\left\Vert u-u_{h}\right\Vert _{h} \leq C\left(\inf_{v_h \in V_h}\vertiii{u-v_{h}}_{h}+\sup_{w_{h} \in V_h}\frac{L\left(w_{h}\right)-A\left(u,w_{h}\right)}{\left\Vert w_{h}\right\Vert_{h}}\right).
	\label{ineq:Strang}
	\end{equation}
	The first term on the right hand side of \eqref{ineq:Strang} can be estimated using \eqref{lemma1}. For the second term of \eqref{ineq:Strang}, recall that $f=-\Delta u$ so we can write the numerator as follows.
	\begin{equation}
	\begin{aligned}
	& \int_{\Omega}	fw_{h}\,dx-\int_{\Omega}P_{h}\left(\nabla u\right)\cdot P_h\left(\nabla w_{h}\right)\,dx+\int_{\Gamma}\left(P_h\left(\nabla u\right)	\cdot\mathbf{n}\right)w_{h}~ds\\
	& =-\int_{\Omega}\Delta uw_{h}\,dx-\int_{\Omega}P_h\left(\nabla u\right)\cdot P_h\left(\nabla w_{h}\right)\,dx+\int_{\Gamma}\left(P_h\left(\nabla u\right)\cdot\mathbf{n}\right)w_{h}~ds\\
	& =\int_{\Omega}\nabla u\cdot\nabla w_{h}\,dx-\int_{\Gamma}\left(\nabla u\cdot\mathbf{n}\right)w_{h}~ds\\
	& \quad -\int_{\Omega}P_h\left(\nabla u\right)\cdot P_h\left(\nabla w_{h}\right)\,dx+\int_{\Gamma}\left(P_h\left(\nabla u\right)\cdot\mathbf{n}\right)w_{h}~ds\\
	& =\int_{\Omega}\nabla u\cdot\nabla w_{h}\,dx-\int_{\Omega}P_h\left(\nabla u\right)\cdot P_h\left(\nabla w_{h}\right)\,dx\\
	& \quad +\int_{\Gamma}\left(P_h\left(\nabla u\right)\cdot\mathbf{n}-\nabla u\cdot\mathbf{n}\right)w_{h}~ds\\
	& =\int_{\Omega}\nabla u\cdot\nabla w_{h}\,dx-\int_{\Omega}\nabla u\cdot P_h\left(\nabla w_{h}\right)\,dx+\int_{\Omega}\nabla u\cdot P_h\left(\nabla w_{h}\right)\,dx\\
	& \quad -\int_{\Omega}P_h\left(\nabla u\right)\cdot P_h\left(\nabla	w_{h}\right)\,dx+\int_{\Gamma}\left(P_h\left(\nabla u\right)\cdot\mathbf{n}-\nabla u\cdot\mathbf{n}\right)w_{h}~ds\\
	& =\int_{\Omega}\nabla u\cdot\left(\nabla w_{h}-P_h\left(\nabla w_{h}\right)\right)\,dx+\int_{\Omega}\left(\nabla u-P_h\left(\nabla u\right)\right)\cdot P_h\left(\nabla w_{h}\right)\,dx\\
	& \quad +\int_{\Gamma}\left(P_h\left(\nabla u\right)\cdot\mathbf{n}-\nabla u\cdot\mathbf{n}\right)w_{h}~ds
	\end{aligned}
	\label{eq:Strang2}
	\end{equation}
	
	We can estimate the first term of \eqref{eq:Strang2} using approximation property of $P_h^{\ast}$ \eqref{ineq:Pestimate} as
	\begin{equation*}
	\begin{aligned}
	\int_{\Omega}\nabla u\cdot\left(\nabla w_{h}-P_h\left(\nabla w_{h}\right)\right)\,dx  & =\int_{\Omega}\left(\nabla u-P_h^{\ast}\left(\nabla u\right)\right)\cdot\left(\nabla w_{h}-P_h\left(\nabla w_{h}\right)\right)\,dx,\\
	& \leq\left\Vert \nabla u-P_h^{\ast}\left(\nabla u\right)\right\Vert_{0,\Omega}\left\Vert \nabla w_{h}-P_h\left(\nabla w_{h}\right)\right\Vert _{0,\Omega},\\
	& \leq Ch\left\Vert u\right\Vert _{2,\Omega}\left\Vert \nabla w_{h}-P_h\left(\nabla w_{h}\right)\right\Vert _{0,\Omega}.
	\end{aligned}
	\end{equation*}
	We can estimate the second term of \eqref{eq:Strang2} using approximation property of $P_h$ \eqref{ineq:Pestimate} as
	\begin{equation*}
	\begin{aligned}
	\int_{\Omega}\left(\nabla u-P_h\left(\nabla u\right)\right)\cdot P_h\left(\nabla w_{h}\right)\,dx  & \leq\left\Vert \nabla u-P_h\left(	\nabla u\right)\right\Vert _{0,\Omega}\left\Vert P_h\left(\nabla	w_{h}\right)\right\Vert _{0,\Omega},\\
	& \leq Ch\left\Vert u\right\Vert _{2,\Omega}\left\Vert P_h\left(\nabla w_{h}\right)\right\Vert _{0,\Omega},
	\end{aligned}
	\end{equation*}
	We can estimate the third term of \eqref{eq:Strang2} as
	\begin{equation*}
	\begin{aligned}
	\int_{\Gamma}\left(P_h\left(\nabla u\right)\cdot n-\nabla u\cdot n\right)w_{h}~ds  & \leq\left\Vert P_h\left(\nabla u\right)\cdot\mathbf{n}-\nabla u\cdot\mathbf{n}\right\Vert _{-1/2,h}\left\Vert	w_{h}\right\Vert _{1/2,h},\\
	& \leq Ch\left\Vert u\right\Vert _{2,\Omega}\left\Vert w_{h}\right\Vert_{1/2,h},
	\end{aligned}
	\end{equation*}
	where 
	\[
	\left\Vert P_h\left(\nabla u\right)\cdot\mathbf{n}-\nabla u\cdot\mathbf{n}\right\Vert _{-1/2,h} \leq Ch\left\Vert u\right\Vert _{2,\Omega},
	\]
	follows from the approximation property of $P_h$ \cite{LW04}.
	Combining all estimates concludes the proof.
\end{proof}

\section{Numerical Examples}

\label{sec:examples}

In this section, we show two numerical examples to verify the convergence rate of our approach. We compute the error in $L^{2}$-norm and the rate of convergence for $u$ and $\sigma $. We also compute the error in $H^{1}$-norm and the rate of convergence for $u$. We will use pure Dirichlet boundary conditions for all our examples.

\subsection*{Example 1}

We consider the exact solution
\begin{equation*}
u=xy\left( 1-x\right) \left( 1-y\right) ,
\end{equation*}%
for the first example. The errors for this example with pure Dirichlet boundary conditions are shown in Table \ref{tab:ex1}.
\begin{table}[htb!]
	\centering
	\caption{Discretisation errors with pure Dirichlet boundary conditions for example 1}
	\begin{tabular}{|c|cc|cc|cc|}
		\hline
		\multirow{2}{*}{elem} & \multicolumn{2}{c|}{$\left\Vert u - u_h\right\Vert_{0,\Omega}$} & \multicolumn{2}{c|}{$\left\Vert u - u_h\right\Vert_{1,h}$} & \multicolumn{2}{c|}{$\left\Vert \sigma - \sigma_h\right\Vert_{0,\Omega}$} \\ \hhline{~------}
		& error & rate & error & rate & error & rate \\\hline
    8 & 3.74e-02 &  & 1.98e-01 &  & 1.73e-01 & \\\hline
   32 & 8.89e-03 & 2.0742 & 1.09e-01 & 0.8654 & 5.94e-02 & 1.5444 \\\hline
  128 & 1.92e-03 & 2.2083 & 5.53e-02 & 0.9754 & 1.81e-02 & 1.7175 \\\hline
  512 & 4.37e-04 & 2.1364 & 2.76e-02 & 1.0035 & 5.68e-03 & 1.6702 \\\hline
 2048 & 1.04e-04 & 2.0752 & 1.37e-02 & 1.0062 & 1.87e-03 & 1.6056 \\\hline
 8192 & 2.52e-05 & 2.0392 & 6.85e-03 & 1.0042 & 6.33e-04 & 1.5590 \\\hline
	\end{tabular}
	\label{tab:ex1}
\end{table}

\subsection*{Example 2}

We consider the exact solution%
\begin{equation*}
u=e^{x^{2}+y^{2}}+y^{2}\cos \left( xy\right) +x^{2}\sin \left( xy\right) ,
\end{equation*}%
for our second example. The errors for this example with pure Dirichlet boundary conditions are shown in Table \ref{tab:ex2}.
\begin{table}[htb!]
	\centering
	\caption{Discretisation errors with pure Dirichlet boundary conditions for example 2}
	\begin{tabular}{|c|cc|cc|cc|}
		\hline
		\multirow{2}{*}{elem} & \multicolumn{2}{c|}{$\left\Vert u - u_h\right\Vert_{0,\Omega}$} & \multicolumn{2}{c|}{$\left\Vert u - u_h\right\Vert_{1,h}$} & \multicolumn{2}{c|}{$\left\Vert \sigma - \sigma_h\right\Vert_{0,\Omega}$}\\ \hhline{~------}
		& error & rate & error & rate & error & rate \\\hline
    8 & 7.36e-01 &  & 4.23e+00 &  & 2.32e+00 & \\\hline
   32 & 1.50e-01 & 2.2902 & 2.10e+00 & 1.0110 & 8.56e-01 & 1.4381 \\\hline
  128 & 3.12e-02 & 2.2694 & 1.03e+00 & 1.0258 & 2.93e-01 & 1.5457 \\\hline
  512 & 6.83e-03 & 2.1915 & 5.07e-01 & 1.0231 & 1.00e-01 & 1.5494 \\\hline
 2048 & 1.57e-03 & 2.1179 & 2.51e-01 & 1.0149 & 3.45e-02 & 1.5384 \\\hline
 8192 & 3.76e-04 & 2.0661 & 1.25e-01 & 1.0085 & 1.20e-02 & 1.5263 \\\hline
	\end{tabular}
	\label{tab:ex2}
\end{table}

From Tables \ref{tab:ex1} and \ref{tab:ex2}, we can see that the rate of convergence of errors for $u$ in $L^2$-norm and $\left(1,h\right)$-norm is 2 and 1, respectively, while the rate of convergence of errors for $\sigma$ in $L^2$-norm is 1.5. These results are very similar to the result from the three-field formulation for Poisson problem with same examples. 

\section{Conclusion}

\label{sec:conclusion}

In this article, we describe a mixed finite element method to solve Poisson equation based on Nitsche's method. We add a stabilisation term so that our bilinear form is coercive on the whole space. From numerical examples, we can observe that the error and rate of convergence is very similar to our previous three-field formulation for Poisson problem. Thus we can conclude that this approach works well as an alternative to the standard formulation. 

\section*{Acknowledgment}
The first author is supported by UNIPRS and UNRSC50:50 scholarship for this research. He would also thank University of Newcastle HDR Funding 2017 and ANZIAM Student Support Scheme for the conference support.


\end{document}